\documentclass[reqno,centertags,12pt]{amsart}
\usepackage{amsmath,amsthm,amscd,amssymb,latexsym,verbatim}

\usepackage{graphicx,epsf,cite}

\usepackage{xcolor}

\usepackage{tikz}
\pgfdeclarelayer{nodelayer}
\pgfdeclarelayer{edgelayer}
\pgfsetlayers{edgelayer,nodelayer,main}
\tikzstyle{arrow}=[draw=black,arrows=-latex]

-\marginparwidth \oddsidemargin 0mm \evensidemargin \oddsidemargin

\newtheorem{theorem}{Theorem}[section]

\newtheorem{lemma}[theorem]{Lemma}

\theoremstyle{definition}

\theoremstyle{remark}

\newcounter{smalllist}

\DeclareMathOperator*{\sgn}{sgn}

\allowdisplaybreaks
\numberwithin{equation}{section}

\newcommand{\lb}{\label}

\newcommand{\beq}{\begin{equation}}
\newcommand{\eeq}{\end{equation}}

\newcommand{\bal}{\begin{align}}
\newcommand{\eal}{\end{align}}
\newcommand{\bals}{\begin{align*}}
\newcommand{\eals}{\end{align*}}

\newcommand{\bbN}{{\mathbb{N}}}
\newcommand{\bbR}{{\mathbb{R}}}

\newcommand{\bbZ}{{\mathbb{Z}}}

\newcommand{\calF}{{\mathcal F}}

\newcommand{\eps}{\varepsilon}

\begin{document}
\title[Stable Regime Singularity for the Muskat Problem]
{The 2D Muskat Problem II: \\ Stable Regime Small Data Singularity on the Half-plane}

\author{Andrej Zlato\v s}
\address{\noindent Department of Mathematics \\ UC San Diego \\ La Jolla, CA 92093, USA \newline Email:
zlatos@ucsd.edu}


\begin{abstract}
We study the Muskat problem on the half-plane, which models motion of an interface between two fluids of distinct densities (e.g., oil and water) in a porous medium (e.g., an aquifer) that sits atop an impermeable layer (e.g., bedrock).  Existence of finite time stable regime interface curve singularities is still open on the whole plane, but we show that they do arise on the half-plane, including from arbitrarily small smooth initial data.  To obtain this result, we establish maximum principles for both the potential energy and the slope of solutions in this model, as well as develop a general local well-posedness theory in the companion paper \cite{ZlaMuskatLocal}.
\end{abstract}

\maketitle

\section{Introduction and Main Results} \lb{S1}

%
%

The {\it Muskat problem}  models motion of the interface between two incompressible immiscible fluids of different densities $\rho_1>\rho_0$ (such as water and oil, or salt water and fresh water) inside a porous medium (such as a  sand or sandstone aquifer) \cite{Muskat, Bear}.    It involves the PDE
\begin{equation} \lb{1.1}
\partial_{t} \rho + u\cdot\nabla \rho =0
\end{equation}
for the fluid density 
\beq\lb{1.3a}
\rho({\bf x},t)=\rho_1- (\rho_1-\rho_0)\chi_{\Omega_t}({\bf x}),
\eeq
where $u$ is the fluid velocity and $\Omega_t$ is the (time-dependent) region of the lighter fluid.
The velocity is  determined via {\it Darcy's law}
\begin{equation} \lb{1.2}
u:=-\nabla p - (0,\rho) \qquad \text{and}\qquad \nabla\cdot u=0,
\end{equation}
 which takes this form after  temporal scaling of \eqref{1.3a} by a factor that is the product of the gravitational constant, permeability of the medium, and the reciprocal of the fluid viscosity.  Here $p$ is the fluid pressure, which is uniquely determined by the incompressibility constraint $\nabla\cdot u=0$, and  $-(0,\rho)$ represents downward motion of denser fluid regions due to gravity (with $0$ the zero vector in one fewer dimensions than ${\bf x}$).
 
The Muskat problem is related to the Hele-Shaw cell problem \cite{SafTay, Hele-Shaw}, and has even been used to  model cell velocity in the growth of tumors \cite{Friedman, Pozrikidis}.   One can also consider it for a single fluid with a continuously varying density, when the model is known as the {\it incompressible porous medium (IPM) equation}. 

We consider here the Muskat problem in two dimensions, where it has been studied extensively on the whole plane $\bbR^2$.  
Local regularity on $\bbR^2$  was proved in \cite{CorGanMuskat} for fluid interfaces that are graphs of  $H^3(\bbR)$   functions of the horizontal variable (see also \cite{SCH,DLL} for small and monotone initial data results), 
provided the lighter fluid lies above the heavier one.  This is the {\it stable regime}, when the Rayleigh-Taylor condition  holds \cite{Rayleigh, SafTay}, and the problem is ill-posed in Sobolev spaces when the heavier fluid lies above the lighter one anywhere along the interface  \cite{CorGanMuskat}.  Numerous other stable regime well-posedness  results on the plane were obtained in the last two decades, including local regularity for arbitrarily large data and global regularity for small data in various subcritical and critical Sobolev spaces.  We only mention here results from \cite{AlaNgu4, Cameron2} in the critical spaces $H^{3/2}(\bbR)$ (for large data) and $\dot W^{1,\infty}(\bbR)$ (for small data in 3D but the result also applies in 2D),
and refer the reader to our companion paper \cite{ZlaMuskatLocal} and reviews \cite{GanRev, GraLaz} for many additional  references.

The limitation to the stable regime does not apply in the space of analytic interfaces \cite{CCFGL}, and it was in fact shown in \cite{CCFGL, CCFG} that there are initial $H^4(\bbR)$ interfaces that instantly become analytic, then develop {\it overturning} in finite time, and later lose analyticity (and even $C^4$ regularity) while remaining in the (unstable) overturned regime (cf. \cite{CGZ2}).  Hence these papers demonstrated existence of overturning in finite time in Sobolev spaces, at which point local regularity breaks down and the interface has a vertical tangent --- while remaining a smooth planar {\it curve} --- as well as (post-overturning)  finite time loss of interface regularity in the space of analytic curves.  Nevertheless, the latter is due to  Rayleigh-Taylor instability when the model is in the unstable regime, rather than a direct result of {\it super-criticality} of the PDE (which is due to $u$ having the same level of regularity as $\rho$).
Possible development of finite time singularities for the Muskat problem on $\bbR^2$ in the stable regime (while the interface slope stays bounded) remains an outstanding open question, despite the instant analyticity result for $H^4(\bbR)$ interfaces from \cite{CCFGL} and other instant smoothing results for less regular initial data, including some with infinite slopes \cite{AlaNgu4} and corners \cite{GGNP, GGHP, Cameron2}.

However, since aquifers typically sit atop impermeable rocky layers, the Muskat problem has a particular physical relevance when considered on the half-plane  $\bbR\times\bbR^+$.  An additional appeal of this setting is that it allows one to  model dynamically important situations when the heavier fluid {\it invades} regions initially occupied exclusively  by the lighter fluid by  
flowing underneath the latter along the impermeable bottom (or vice versa when the aquifer lies below an impermeable layer).  The local regularity proof in $H^3(\bbR)$ from  \cite{CorGanMuskat} (and others) does extend to the half plane if the interface stays  uniformly away from the bottom, but constants in the available estimates blow up as the interface approaches $\bbR\times\{0\}$.  Hence this does not allow one to extend the theory to the most interesting setting when the interface touches the bottom, which includes the above invasion scenario.

Nevertheless, it turns out that local regularity for $H^3$ interfaces does hold on the half-plane, and we provide a (much more involved) proof of this in the companion paper \cite{ZlaMuskatLocal}.  There we also show that the setting of zero contact angles between the interface and the boundary, resulting from the $H^3$ level of regularity (prior to any singularity development), is in fact the correct one for the Muskat dynamic.
The obvious next question is emergence of  finite time interface singularities in the stable regime in this setting (and in particular within the well-posedness theory in Sobolev spaces),  with a possible difference form $\bbR^2$  being the solution dynamic at the boundary.

The main result of the present paper is Theorem \ref{T.1.1} below, which shows that {\it finite time stable regime interface singularities do indeed develop} for the Muskat problem on the half-plane.  Moreover, this can happen for  arbitrarily small smooth initial data (in fact, it happens for all  data from a fairly general class) and after an arbitrarily long time.  The mechanism for this is essentially the invasion scenario described above,  specifically when the invasion happens from both directions, and  the singularity likely develops at the meeting point of the two invading ``fronts''.

Our proof has three main ingredients.  The first is the local regularity result from \cite{ZlaMuskatLocal}, which in particular shows that if the interface touches the bottom initially, it cannot ``peel off'' the bottom unless a singularity occurs.  The second is a {\it quantitative $L^2$ maximum principle} for the fluid interface in Theorem \ref{T.1.3}, which shows that unless the interface is close to flat (except on a small set), its $L^2$ norm (when it is appropriately defined and finite) decreases at a uniformly positive rate.  Such a maximum principle is known for the Muskat problem on $\bbR^2$ \cite{CCGS}, and since this $L^2$ norm is also the potential energy of the interface, one can even think of this result as monotonicity of the latter.  This was recently derived for the IPM equation in \cite{KisYao}, and used there to prove that solutions must be asymptotically stratified.  Since topological constraints force Sobolev norms of some solutions to grow as they stratify (unless they become singular), \cite{KisYao} also demonstrated unbounded growth of certain solutions to IPM on a horizontal strip.

The same conclusion can be derived for the Muskat problem on the half plane, once the local well-posedness theory from \cite{ZlaMuskatLocal} and Theorem \ref{T.1.3} are available, because mass conservation forces interfaces that  become more and more flat while still touching the bottom to  have near-vertical tangent lines somewhere.  But just as in \cite{KisYao}, this is not sufficient to prove development of a singularity (of the interface as a curve) or even of interface overturning (while it remains a smooth curve)  in finite time.

Instead, we conclude the proof of existence of finite time singularity using our third ingredient, a maximum principle for the slope of the interface on the half-plane from Theorem~\ref{T.1.2}(ii) below.  We show that if the initial interface is not too steep, then it remains such as long as it remains regular and so its tangents cannot become near-vertical.  It follows that a finite time interface singularity must occur, with no overturning and so in the stable regime, for any initial interface (from an appropriate class) that is  not too steep and  touches the bottom.  It is possible that some version of our approach could be also applicable to the question of finite time singularity for IPM on the half-plane or on horizontal strips, although this will first require extension of the local well-posedness theory to solutions with more general boundary behaviors than what is currently available \cite{CCL}.

We note that a simple proof of a maximum principle for interface slopes on the whole plane was provided in \cite{CorGanMuskatMax} (see also \cite{DLL}), but the half-plane version is much more involved and its proof in fact forms the bulk of this paper.  To illustrate the challenges involved, we note that a version of it on strips $\bbR\times(0,l)$ was obtained in \cite{CorGraOri}, but this result has fairly restrictive hypotheses and only applies to interfaces that are {\it far away} from the strip boundaries.  In contrast, Theorem \ref{T.1.2}(ii) applies to all interfaces on the half-plane with slopes in $[-\frac 3{10},\frac 3{10}]$.

Since this result is of independent interest, we prove it in more generality than we need here,
namely for interfaces that need not flatten as $|x|\to \infty$ or be periodic, as was the case in all previous local regularity results except for the small data result in \cite{Cameron2}.  We even allow $O(|x|^{1-})$ growth of the interface, which is the optimal power for a general local well-posedness result (see \cite{ZlaMuskatLocal}), and one of the contributions of \cite{ZlaMuskatLocal} is extension of existing theory for the Muskat problem to such interfaces (on the plane, half-plane, and horizontal strips).  

%

\vskip 3mm
\noindent
{\bf Basic setup and local well-posedness.} 
From \eqref{1.2} we see that the  {\it vorticity} $\nabla^\perp\cdot u$ equals $\rho_{x_1}$ if we define $\nabla^\perp:=(\partial_{x_2},-\partial_{x_1})$.  Since we have the {\it no-flow boundary condition} $u_2\equiv 0$ on  $\bbR\times\{0\}$, it follows that
\begin{equation} \lb{1.3}
u({\bf x},t):= \nabla^\perp \Delta^{-1} \rho_{x_1}({\bf x},t) = \frac 1{2\pi} \int_{\bbR\times\bbR^+} \left( \frac{({\bf x}-{\bf y})^\perp}{|{\bf x}-{\bf y}|^2} - \frac{({\bf x}-\bar {\bf y})^\perp}{|{\bf x}-\bar {\bf y}|^2} \right) \rho_{x_1}({\bf y},t) \, d{\bf y},
\end{equation}
where $\Delta$ is the Dirichlet Laplacian on $\bbR\times\bbR^+$, $\bar {\bf y}:=(y_1,-y_2)$, and ${\bf y}^\perp:=(y_2,-y_1)$.
 Motion of the fluid interface   $\partial\Omega_t$ (see \eqref{1.3a}) determines the full dynamic for the Muskat problem.  One can use this to derive from \eqref{1.1}--\eqref{1.2} an equivalnet PDE for the fluid interface $x_1\mapsto f(x_1,t)\ge 0$, with $\Omega_t=\{x_2>f(x_1,t)\}$ for $t\ge 0$ when the fluids are in the stable regime.  This is done in \cite{ZlaMuskatLocal} via the same argument as on the plane, and we get
\begin{equation} \lb{1.5}
f_t(x,t) = PV \int_\bbR \left[ \frac{y\, (f_x(x,t)-f_x(x-y,t))}{y^2+(f(x,t)-f(x-y,t))^2} + \frac{y\, (f_x(x,t)+f_x(x-y,t))}{y^2+(f(x,t)+f(x-y,t))^2} \right]  dy,
\end{equation}
where we also assumed without loss that $\rho_1-\rho_0=2\pi$.  The corresponding whole plane version is instead
\begin{equation} \lb{1.13}
f_t(x,t) = PV \int_\bbR  \frac{y\, (f_x(x,t)-f_x(x-y,t))}{y^2+(f(x,t)-f(x-y,t))^2}  dy.
\end{equation}
 
Before stating our main results, let us first introduce the general spaces of interface functions that we will consider in Theorem~\ref{T.1.2}.
For  $k=0,1,\dots $ define the (local) norms
\begin{align*}
\|g\|_{\tilde L^2(\bbR)}  := \sup_{x\in\bbR} \|g\|_{L^2([x-1,x+1]) } 
\qquad\text{and}\qquad
 \|g\|_{\tilde H^k(\bbR)}  := \sum_{j=0}^k \|g^{(j)}\|_{\tilde L^2(\bbR)} 
\end{align*}
on the spaces of those $g\in L^2_{\rm loc}(\bbR)$ for which these are finite. Then for  $k\ge 1$ and any $\gamma\in[0,1]$ define the seminorms
\begin{align*}
 \|g\|_{\ddot C^\gamma(\bbR)}  := \sup_{|x- y|\ge 1} \frac{|g(x)-g(y)|}{|x-y|^\gamma} 
 \qquad\text{and}\qquad
 \|g\|_{\tilde H^k_\gamma(\bbR)}  :=  \|g'\|_{\tilde H^{k-1}(\bbR)} 
+ \|g\|_{\ddot C^{1-\gamma}(\bbR)} ,
\end{align*}
which vanish for constant functions.  
We note that $\|g\|_{\tilde H^k_\gamma(\bbR)}<\infty$ allows $g$ to have $O(|x|^{1-\gamma})$ growth at $\pm\infty$, and refer the reader to  \cite{ZlaMuskatLocal} for further discussion.  We now have the following local regularity result, 
which holds on both the plane and the half-plane.

%

\begin{theorem}[\hskip0.01mm \cite{ZlaMuskatLocal}] \lb{T.1.1X}
Let $\gamma\in(0,1]$ and $\|\psi\|_{\tilde H^3_\gamma(\bbR)}<\infty$.

(i)  
There is $\gamma$-independent $T_\psi\in(0,\infty]$  and a unique classical solution $f$ to \eqref{1.13} on $\bbR\times[0,T_\psi)$ with $f(\cdot,0)\equiv\psi$
such that 
\beq \lb{1.7}
\sup_{t\in[0,T]} \|f(\cdot,t)\|_{\tilde H^3_\gamma(\bbR)} <\infty
\eeq
for each $T\in(0,T_\psi)$.  It in fact satisfies
\beq \lb{1.7a}
\sup_{t\in[0,T]} \|f(\cdot,t)-\psi\|_{\tilde H^3(\bbR)} <\infty 
\qquad\text{and}\qquad \sup_{t\in[0,T]} \|f_t(\cdot,t)\|_{W^{1,\infty}(\bbR)}<\infty
\eeq
for each $T\in(0,T_\psi)$, and if there are $\psi_{\pm\infty}\in\bbR$ such that $\psi-\psi_{\pm\infty}\in H^3(\bbR^\pm)$, then even 
\beq \lb{1.11}
\sup_{t\in[0,T]}  \|f(\cdot,t)-\psi\|_{H^3(\bbR)}  <\infty.
\eeq
Finally, if $T_\psi<\infty$, then for each $\gamma'\in(0,1]$ we have
\beq \lb{1.8}
\int_0^{T_\psi} \|f_x(\cdot,t)\|_{C^{1,\gamma'}(\bbR)}^4 \, dt =\infty.
\eeq

(ii)
If $\psi\ge 0$, then (i) holds with $f\ge 0$ solving  \eqref{1.5}  instead.  And if $\inf \psi=0$, then $\inf f(\cdot,t)=0$ for each $t\in[0,T_\psi)$.
\end{theorem}

\vskip 3mm
\noindent
{\bf Main results.} 
We are now ready to state our main result, which shows that for the Muskat problem on the half-plane, stable regime finite time singularity develops from all initial data that touch the bottom and do not have large slopes (plus they are either periodic or $H^3(\bbR)$ perturbations of non-zero constant functions).  In it we also use the seminorm
\[
 \|g\|_{\dot C^\gamma(\bbR)}  := \sup_{x\neq y} \frac{|g(x)-g(y)|}{|x-y|^\gamma} .
\]
  

\begin{theorem} \lb{T.1.1}
If $0 \le  \psi\in \tilde H^3(\bbR)$ is periodic, $\min \psi=0\not \equiv \psi$, and $\|\psi'\|_{L^\infty(\bbR)}\le \frac 3{10}$, then $T_\psi<\infty$ and $f$ from Theorem \ref{T.1.1}(ii) satisfies 
\beq \lb{1.8a}
\|f_x\|_{L^\infty(\bbR\times[0,T_\psi))}\le \frac 3{10} \qquad\text{and}\qquad
\int_0^{T_\psi} \|f_{xx}(\cdot,t)\|_{\dot C^{\gamma'}(\bbR)}^4 \, dt =\infty
\eeq
for each $\gamma'\in(0,1]$.  The same result holds when periodicity of $\psi$ is replaced by $\psi-\psi_\infty\in  H^3(\bbR)$ for some constant $\psi_\infty\in(0,\infty)$.
\end{theorem}

{\it Remarks.}
1.  Since $\psi\in \tilde H^3(\bbR)$, we have $\|\psi\|_{\tilde H^3_\gamma(\bbR)}<\infty$ for all $\gamma\in[0,1]$.  
This and the first claim in \eqref{1.7a} show that  we do not need to specify $\gamma$ in Theorem \ref{T.1.1}.
\smallskip

2.  So  while remaining a Lipschitz graph, the interface leaves $C^{2,\gamma'}(\bbR)$ for each $\gamma'>0$ (and hence also $H^3_{\rm loc}(\bbR)$) at time $T_\psi<\infty$ when, for instance, for some (arbitrarily) small $\eps>0$ we have $\|\psi-\eps\|_{H^3(\bbR)}\le C \eps$.  And this can happen after an arbitrarily long time because in \cite{ZlaMuskatLocal} we showed that
$T_\psi \ge C_\gamma (1 - \ln \|\psi\|_{\tilde H^3_\gamma(\bbR)})$ for some $C_\gamma>0$.
\smallskip

As we explained above, the main ingredient in the proof of Theorem \ref{T.1.1} is the following $L^\infty$ maximum principle for $f_x$ on the half-plane.  Since we prove it in the general case $\|\psi\|_{\tilde H^3_\gamma (\bbR)}<\infty$, the result is new even on the plane and we state it here on both domains.

\begin{theorem} \lb{T.1.2}
(i)
Let $\gamma,\psi,f$ be as in Theorem \ref{T.1.1X}(i). If $\|f_x(\cdot,t')\|_{L^\infty(\bbR)} \le 1$ for some $t'\in[0,T_\psi)$, then $\|f_x(\cdot,t)\|_{L^\infty(\bbR)}$ is decreasing on $[t',T_\psi)$ (strictly unless it is 0).

(ii)
If $\gamma,\psi,f$ are as in Theorem \ref{T.1.1X}(ii), then (i) holds with $\|f_x(\cdot,t')\|_{L^\infty(\bbR)} \le \frac 3{10}$ instead. 
\end{theorem}

Finally, we state the $L^2$ maximum principle for $f$ (which is new on the half-plane) that we also use in the proof of Theorem~\ref{T.1.1}.
In it we let
\[
\lambda(a,b,c) := \frac {(a+ b)^2 (a-b)^2} { 2 [c^2+(a+b)^2] \, [c^2+(a-b)^2]} \ge 0.
\]

\begin{theorem} \lb{T.1.3}
Let $0 \le  \psi\in \tilde H^3(\bbR)$ and let $f$ be from Theorem \ref{T.1.1}(ii).

(i)
If $\psi$ is periodic with period $\nu$, then 
\beq\lb{1.9}
 \frac d{dt}\|f(\cdot,t)\|_{L^2([0,\nu])}^2 = -2\|\rho_{x_1}(\cdot,t)\|_{\dot H^{-1} ([0,\nu]\times\bbR^+)}^2 
 \le -    \int_{[0,\nu]\times\bbR} \lambda \big( f(x,t),f(y,t),x-y \big) \, dxdy
\eeq
 for each $t\in[0,T_\psi)$, where $\rho(\cdot,t):=2\pi \chi_{\{0<x_2<f(x_1,t)\}}$.

(ii)
If  $\psi-\psi_\infty\in  H^3(\bbR)$ for some $\psi_\infty\in(0,\infty)$, then  $\tilde f:=f-\psi_\infty$ satisfies
\beq\lb{1.9a}
 \frac d{dt}\|\tilde f(\cdot,t)\|_{L^2(\bbR)}^2 = -2\|\rho_{x_1}(\cdot,t)\|_{\dot H^{-1} (\bbR\times\bbR^+)}^2 
 \le -    \int_{\bbR^2} \lambda \big( f(x,t),f(y,t),x-y \big) \, dxdy
\eeq
for each $t\in[0,T_\psi)$, where $\rho(\cdot,t):=2\pi \chi_{\{0<x_2<f(x_1,t)\}}$.
\end{theorem}

{\it Remark.}
The proof of Theorem \ref{T.1.3} shows that for \eqref{1.13} we instead have
\[
 \frac d{dt}\|\tilde f(\cdot,t)\|_{L^2(S)}^2 = - \int_{S\times\bbR}\ln \left(1+ \frac{(f(x,t)- f(y,t))^2}{(x-y)^2} \right) dxdy,
\]
where $S$ is $[0,\nu]$ in (i) (with $\psi_\infty:=0$) and it is $\bbR$ in (ii).  This is a result  from \cite{CCGS}.
\smallskip

The rest of the paper is composed of four sections.  We prove Theorem \ref{T.1.1} in Section \ref{S5}, using Theorems \ref{T.1.2} and \ref{T.1.3}.  The latter result is then proved in Section \ref{S4}, while the much longer proof of Theorem \ref{T.1.2} is split between Sections \ref{S2} and \ref{S3}.

\section{Proof of Theorem \ref{T.1.1}} \lb{S5}

The proofs of both claims are similar, so we only provide here the first one.  We note that the proof of the second claim involves  norms $\|\tilde f(\cdot,t)\|_{L^1(\bbR)}$ and $\|\tilde f(\cdot,t)\|_{L^2(\bbR)}$ with $\tilde f:=f-\psi_\infty$ instead of $\|f(\cdot,t)\|_{L^1([0,1])}$ and $\|f(\cdot,t)\|_{L^2([0,1])}$ (see \eqref{1.11} and the end of the next section).

Assume that the period of $\psi$ is 1 because the general case is identical.  By the uniqueness claim in Theorem \ref{T.1.1X}(i), $f(\cdot,t)$ also has period 1 for each $t\in[0,T_\psi)$.  

Theorem \ref{T.1.2}(ii) proves the first claim in \eqref{1.8a}, while \eqref{4.0a} below shows that 
\[
\|f(\cdot,t)\|_{L^1([0,1])} = \|\psi\|_{L^1([0,1])} =: L \in \left(0,\frac 3{40} \right)
\]
for all $t\in[0,T_\psi)$. 
The last claim in Theorem \ref{T.1.1X}(ii) and the first claim in \eqref{1.8a} now show that for each $t\in[0,T_\psi)$, there is $A_t\subseteq[0,1]$  that is an interval or a union of two intervals of total length $\ge 2L$ such that $f(\cdot,t)\le \frac L3$ on $A_t$.  These two claims also show that $\|f\|_{L^\infty(\bbR\times[0,T_\psi))}\le \frac 3{20}$, and so there is $B_t\subseteq[0,1]$  that is an interval or a union of two intervals of total length $\ge 2L$ such that $f(\cdot,t)\ge \frac {2L} 3$ on $B_t$.  But then the  integral in \eqref{1.9} is at least
\[
4L^2\frac{ \frac {4L^2}9\,\frac {L^2}9}{2 \left(1+(\frac {L}3+\frac 3{20})^2 \right) \left( 1+\frac 9{400} \right)} \ge \frac {L^6}{20}
\]
for each $t\in[0,T_\psi)$.  This and \eqref{1.9} now imply that $T_\psi<\infty$ because $\|f(\cdot,t)\|_{L^2([0,1])}^2\ge 0$.  If we then combine the first claim in \eqref{1.8a} with \eqref{1.8}, we obtain the second claim in \eqref{1.8a}.

So to complete the proof, it remains to prove Theorems \ref{T.1.2} and \ref{T.1.3}, as well as \eqref{4.0a}.

\section{Proof of Theorem \ref{T.1.3}} \lb{S4}

\vskip 3mm
\noindent
{\bf Proof of (i).}  
Assume  that the period $\nu=1$ because the general case is identical.
By the uniqueness claim in Theorem \ref{T.1.1X}(i), $f(\cdot,t)$ also has period 1 for each $t\in[0,T_\psi)$.
Since
\beq \lb{4.0}
 \frac{y\, ( f'(x)\pm f'(x-y))}{y^2+(f(x)\pm f(x-y))^2} 
=  \frac d{dx} \arctan  \frac{ f(x)\pm f(x-y)} y 
\eeq
for any $y\in\bbR$, we see that
\beq \lb{4.0a}
\frac d{dt}\|f(\cdot,t)\|_{L^1([0,1])}=0
\eeq
for  $t\in[0,T_\psi)$.  From \eqref{4.0} and integration by parts in $x$ 
we obtain (dropping $t$ in the notation)
\begin{align*}
\frac 12 \frac d{dt}\|f\|_{L^2([0,1])}^2 & = - \sum_\pm PV \int_{[0,1]\times\bbR}  f'(x)  \arctan  \frac{ f(x)\pm f(x-y)} y \, dx dy
\\ & = - \sum_\pm   \int_{[0,1]^2} PV \sum_{n\in\bbZ} f'(x)  \arctan  \frac{ f(x)\pm f(x-y)} {y+n} \, dx dy
\\ & = -  \sum_\pm \int_{[0,1]^2} PV \sum_{n\in\bbZ} f'(x)  \arctan  \frac{ f(x)\pm f(y)} {x-y-n} \, dx dy
\\ & = - \sum_\pm PV \int_{[0,1]\times\bbR}  f'(x)  \arctan  \frac{ f(x)\pm f(y)} {x-y} \, dx dy.
\end{align*}
Let us write the integral in $x$ as
\beq \lb{4.1}
\int_{0}^{1} \left[ (x-y) \frac d{dx} G \left( \frac{ f(x)\pm f(y)} {x-y} \right) +   \frac{ f(x)\pm f(y)} {x-y} 
 \arctan  \frac{ f(x)\pm f(y)} {x-y} \right] dx,
\eeq
where
\[
G(s):=s\arctan s- \ln \sqrt{1+s^2} = \int_0^s \arctan r\,dr.
\]
We integrate the first term in \eqref{4.1} by parts in $x$ again 
and find that \eqref{4.1} equals
\[
 \int_{0}^{1}  \frac 12 \ln \left(1+ \frac{(f(x)\pm f(y))^2}{(x-y)^2} \right)\, dx + (1-y)  G \left( \frac{ f(1)\pm f(y)} {1-y} \right) - (- y)  G \left( \frac{ f(0)\pm f(y)} {-y} \right)
\]
plus   the additional ``boundary'' term $-  2\pi f(y)$ when $y\in(0,1)$ and $\pm$ is $+$  (from the singularity at $x=y$).
The last two terms above cancel with the same terms when $y$ is increased and decreased by 1 because $f$ is 1-periodic, so we obtain
\beq\lb{4.2}
 \frac d{dt}\|f\|_{L^2}^2 
  = 4\pi \int_0^1 f(y)\, dy -  \sum_\pm \int_{[0,1]\times\bbR} \ln \left(1+ \frac{(f(x)\pm f(y))^2}{(x-y)^2} \right)  dx   dy.
\eeq

We now use the formulas
\[
\ln(1+p)+\ln(1+q)- \ln(1+p+q) \ge \int_1^{1+q} \left( \frac 1z - \frac 1{p+z}\right)dz \ge \frac {pq}{(1+q)(1+p+q)}
\]
with $p:=\frac{(f(x)+ f(y))^2}{(x-y)^2}$ and $q:= \frac{(f(x)- f(y))^2}{(x-y)^2}$, and then
\[
2\ln(1+p+q)-\ln(1+p+q+r)-\ln(1+p+q-r) \ge 0
\]
with $r:=\frac{2f(x)^2- 2f(y)^2}{(x-y)^2}$ to get
\[
2\sum_\pm \ln \left(1+ \frac{(f(x)\pm f(y))^2}{(x-y)^2} \right) \ge \ln \left(1+ \frac{4f(x)^2}{(x-y)^2} \right) +  \ln \left(1+ \frac{4f(y)^2}{(x-y)^2} \right) +2 \tilde \lambda(f(x),f(y),x-y),
\]
where
\beq\lb{4.3}
 \tilde \lambda(a,b,c):= \frac {(a+ b)^2 (a- b)^2} { [c^2+2a^2+2b^2] \, [c^2+(a-b)^2] } \ge 0.
\eeq
Since also
\[
\int_\bbR  \ln \left(1+ \frac{C^2}{y^2} \right) \,dy 
= \int_\bbR \frac d{dy} \left[ y\ln \left(1+ \frac{C^2}{y^2} \right) -2C\arctan \frac Cy \right] dy = 2\pi C
\]
and
\[
\int_{[0,1]^2} \sum_{n\in\bbZ} \ln \left(1+ \frac{4f(y+n)^2}{(x-(y+n))^2} \right) dxdy
 = \int_{[0,1]^2}  \sum_{n\in\bbZ} \ln \left(1+ \frac{4f(y)^2}{((x-n)-y)^2} \right) dxdy
\]
for $y\in(0,1)$ (by 1-periodicity of $f$),
this and \eqref{4.2} yield
\beq\lb{4.4}
 \frac d{dt}\|f\|_{L^2([0,1])}^2 \le -   \int_{[0,1]\times\bbR}   \tilde \lambda(f(x),f(y),x-y) \, dxdy.
 \eeq
Now \eqref{1.9} follows from $\tilde \lambda\ge \lambda$.

%
 
Let $\Gamma_t:=\Gamma_f(t)$ and
\[
\Omega_t':= \big\{ {\bf x}\in(0,1)\times(0,\infty) \,\big|\, x_2< f(x_1,t) \big\}.
\]
 We have
 \[
 \int_0^1 f(x) \partial_t f(x) dx = \int_{\Gamma_t} \frac {x_2\, \partial_t f(x_1)}{\sqrt{1+f'(x_1)^2}} \, d\sigma({\bf x})
 = \int_{\Gamma_t} {x_2 \,(0,\partial_t f(x_1))}  \cdot \frac {(-f'(x_1),1)} {\sqrt{1+f'(x_1)^2}} \, d\sigma({\bf x}) .
 \]
We now use the fact that $\partial_t f(x_1)$ is precisely the number such that at time $t$ we have
 \[
 {(0,\partial_t f(x_1))}  \cdot \frac {(-f'(x_1),1)} {\sqrt{1+f'(x_1)^2}} = u({\bf x})  \cdot \frac {(-f'(x_1),1)} {\sqrt{1+f'(x_1)^2}} ,
 \]
where the second vector on both sides is the unit outer normal $n({\bf x})$ to $\Omega_t'$ at ${\bf x}=(x_1,x_2)$.  So 1-periodicity of $f$, the divergence theorem, and $\nabla\cdot u\equiv 0$ yield
 \begin{align*}
 \frac 12\,  \frac d{dt}\|f\|_{L^2([0,1])}^2 &   = \int_{\partial\Omega_t'} x_2 u({\bf x}) \cdot n({\bf x})\, d{\bf x}  = \int_{\Omega_t'} u_2({\bf x}) \, d{\bf x} = \int_{[0,1]\times[0,\infty)} \rho({\bf x}) u_2({\bf x})  \, d{\bf x}
\\ & =  \int_{[0,1]\times[0,\infty)} \rho({\bf x}) \, \partial_{x_1} (-\Delta)^{-1} \rho_{x_1}({\bf x},t)
= -\|\rho_{x_1}\|_{\dot H^{-1} ([0,1]\times[0,\infty))}^2.
 \end{align*}
The integration by parts at the end can be justified by approximating $\rho$ by smooth characteristic functions $\rho_n$ of $\Omega_t'$ with $u^n:= \partial_{x_1} (-\Delta)^{-1}(\rho_n)_{x_1}$, and then $\rho_n\to\rho$ in $L^2([0,1]\times[0,\infty))$ implies $u^n\to u$ in $L^2( [0,1]\times[0,\infty))$ and $(\rho_n)_{x_1}\to\rho_{x_1}$ in $\dot H^{-1}( [0,1]\times[0,\infty))$.

\vskip 3mm
\noindent
{\bf Proof of (ii).}  
This proof is almost identical to the one above, with all integrals over $[0,1]$ being now over $\bbR$ (which simplifies some steps) and with the whole argument estimating instead
\[
\frac d{dt}\|\tilde f(\cdot,t)\|_{L^1(\bbR)} \qquad\text{and}\qquad \frac d{dt}\|\tilde f(\cdot,t)\|_{L^2(\bbR)}^2.
\]
Since $\tilde f_x=f_x$ and $\tilde f_t=f_t$, all steps in the proof easily extend to this setting.  We note that in the last part of the proof we use
\begin{align*}
\Omega_t' &:= \big\{ {\bf x}\in\bbR\times(0,\infty) \,\big|\, x_2< f(x_1,t) \big\},
\\ \Omega_t^\pm &:= \big\{ {\bf x}\in\bbR\times(0,\infty) \,\big|\, \pm \psi_\infty < \pm x_2< \pm f(x_1,t) \big\}, 
\end{align*}
and also $x_2-\psi_\infty$ in place of $x_2$.  This then yields
 \begin{align*}
\frac 12\,  \frac d{dt}\|f\|_{L^2(\bbR)}^2   & = \int_{\partial\Omega_t^+} (x_2-\psi_\infty) u({\bf x}) \cdot n({\bf x})\, d{\bf x}
- \int_{\partial\Omega_t^-} (x_2-\psi_\infty) u({\bf x}) \cdot n({\bf x})\, d{\bf x}
\\ & = \int_{\Omega_t^+} u_2({\bf x}) \, d{\bf x} - \int_{\Omega_t^-} u_2({\bf x}) \, d{\bf x} = PV \int_{\Omega_t'} u_2({\bf x}) \, d{\bf x}
 \end{align*}
 because $PV \int_{\bbR\times[0,\psi_\infty]} u_2({\bf x}) \, d{\bf x}=0$.

\section{Proof of Theorem \ref{T.1.2}} \lb{S2}

For  $k=0,1,\dots $ and $\gamma\in[0,1]$ let 
\begin{align*}
 \|g\|_{C^{k,\gamma}(\bbR)}  :=  \sum_{j=0}^k \|g^{(j)}\|_{L^\infty(\bbR)}+ \|g^{(k)}\|_{\dot C^{\gamma}(\bbR)}
\end{align*}
when $g\in L^2_{\rm loc}(\bbR)$
(note that then $C^{k,0}(\bbR)=W^{k,\infty}(\bbR)$),
and for $k\ge 1$ define the seminorms (vanish on constant functions) 
\begin{align*}
 \|g\|_{C^{k,\gamma}_\gamma(\bbR)}  := \|g'\|_{C^{k-1,\gamma}(\bbR)} 
+ \|g\|_{\ddot C^{1-\gamma}(\bbR)}  
 \qquad\text{and}\qquad
\|g\|_{C^{k}_\gamma(\bbR)}:=\|g'\|_{C^{k-1,0}(\bbR)} 
+ \|g\|_{\ddot C^{1-\gamma}(\bbR)} .
\end{align*}
Note that when $\gamma\in(0,\frac 12]$ 
and $k\ge 1$, then
\beq\lb{2.0aa}
\|g\|_{C^{k-1,\gamma}_\gamma(\bbR)}\le C_{k,\gamma} \|g\|_{\tilde H^{k}_\gamma(\bbR)}
\eeq
holds for all $g\in L^2_{\rm loc}(\bbR)$, with some constant $C_{k,\gamma}<\infty$.  We will assume without loss that $\gamma\in(0,\frac 12]$ because $\|g\|_{\tilde H^3_\gamma(\bbR)}<\infty$ implies $\|g\|_{\tilde H^3_{\gamma'}(\bbR)}<\infty$ for all $\gamma'\in(0,\gamma]$.

Since
\[
\frac{ f(x)\pm f(x-y) \pm y\, f'(x-y)}{y^2+(f(x)\pm f(x-y))^2}  
= - \frac d{dy} \arctan  \frac{ f(x)\pm f(x-y)} y ,
\]
we see that
\[
 \int_{\bbR}  \frac{ f(x)\pm f(x-y) \pm y\, f'(x-y)}{y^2+(f(x)\pm f(x-y))^2}  \, dy=
 \begin{cases}
 \pi & \text{$\pm$ is $+$ and $f(x)>0$,}
 \\ 0 &  \text{$\pm$ is $-$ or $f(x)=0$}
 \end{cases}
\]
when $\| f\|_{C^2_\gamma(\bbR)}<\infty$ (and $f\ge 0$ when $\pm$ is $+$).
Hence \eqref{1.13} can be equivalently written as
\begin{equation} \lb{1.6a}
f_t(x,t) =  PV \int_\bbR  \frac{(x-y) f_x(x,t)-(f(x,t) - f(y,t))}{(x-y)^2+(f(x,t) - f(y,t))^2}  dy ,
\end{equation}
where we also changed variables $y\leftrightarrow x-y$, while \eqref{1.5} is equivalent to
\begin{equation} \lb{1.6}
f_t(x,t) = \pi \chi_{(0,\infty)}(f(x,t)) +  \sum_\pm PV \int_\bbR  \frac{(x-y) f_x(x,t)-(f(x,t)\pm f(y,t))}{(x-y)^2+(f(x,t)\pm f(y,t))^2}  dy .
\end{equation}
Note that the last integral with $\pm$ being $+$ is discontinuous at all points in $\partial\{x\,|\, f(x,t)=0\}$, but this will not cause a problem because our analysis of \eqref{1.6} will be performed at points where $f_x(x,t)\neq 0$ and therefore $f(x,t)>0$.
Note also that ``PV'' is only needed in \eqref{1.6a} and \eqref{1.6}  as $|y|\to\infty$ but not as $|y|\to 0$ when $\|f(\cdot,t)\|_{C^2_\gamma(\bbR)}<\infty$.


Finally, we observe that the second claim in \eqref{1.7a} implies
\[
\sup_{x\in\bbR} \|f_x(x,\cdot)\|_{W^{1,\infty}([0,T])}\le \|f_x\|_{L^\infty(\bbR\times [0,T])} + \sup_{t\in[0,T]} \|f_t(\cdot,t)\|_{W^{1,\infty}(\bbR)} <\infty
\]
for each $T\in[0,T_\psi)$. This and \eqref{1.7} yield $f_x\in W^{1,\infty}(\bbR\times[0,T])$,
so $M_t:= \|f_x(\cdot,t)\|_{L^\infty}$ is locally Lipschitz on $[0,T_\psi)$.


\vskip 3mm
\noindent
{\bf Proof of (i).} 
Let $M_t$ be as above and assume that $M_t\in(0,1]$ for some $t$ (if $M_t=0$, then $f(\cdot,t)\equiv C$ for some $C\in\bbR$, and then this holds for all $t'\ge t$).
For the sake of simplicity, in the following argument we drop $t$  from the notation (other than in $M_t$) and denote $g':=g_x$.  

For any $(x,t)\in\bbR\times[0,T_\psi)$ with $f'(x)\neq 0$ we have
\[
\begin{split}
\partial_t f'(x) &= \frac d{dx}  PV \int_\bbR   \frac{(x-y) f'(x)-(f(x) - f(y))}{(x-y)^2+(f(x) - f(y))^2} \, dy 
\\ & = PV \int_\bbR   \frac{(x-y) f''(x)}{(x-y)^2+(f(x) - f(y))^2} \, dy  
-  2{f'(x)} PV \int_\bbR  h_{f'(x)}(y, f(x)-f(x-y))  \, dy,
\end{split}
\]
where
\[
h_a(y, r):= \frac{y^2 - r^2 - (a^{-1}-a) yr }{ (y^2+r^2)^2 }
\]
and so the last integral is just
\[
\frac 1{f'(x)} PV \int_\bbR  \frac{[(x-y) f'(x)-(f(x)- f(y))] \, [x-y + f'(x)(f(x)- f(y))] }{ [(x-y)^2+(f(x)- f(y))^2]^2 } \, dy
\]
after changing variables $y\leftrightarrow x-y$.  We note that \eqref{22.32} and \eqref{22.33} below show that both integrals converge in the principal value sense.

When $|f'(x)|=M_t\in(0, 1]$, then $f''(x)=0$ and $h_{f'(x)}(y, f(x)-f(x-y))\ge 0$ for all $y\in\bbR$ (see \eqref{22.29} below), with the inequality being sharp when $|y|$ is large (since equality only holds when $f(x)-f(x-y)=f'(x)(x-y)$).  Hence $\partial_t f'(x)<0$, which yields the result when such $x$ exists for each $t\in[0,T_\psi)$.  But this need not be the case when only $\|f\|_{\tilde H^3_\gamma(\bbR)}<\infty$.  One could circumvent this issue by looking at points where $|f'(x)|$ is maximal on $[x-L,x+L]$ for large $L$, but this will not be the case for the much more involved argument in (ii) because an estimate like \eqref{22.29} is not available for the analogous term coming from the second integral in \eqref{1.6}.  We will therefore instead estimate $\partial_t f'(x)$ at all points where $|f'(x)|$ is close to $M_t$, which will also work in (ii).  Our argument is based on replacing $h_{f'(x)}(y, f(x)-f(x-y))$ by $h_{\sgn(f'(x)) M_t}(y, f(x)-f(x-y))$ and estimating the difference, and then using that $h_{\sgn(f'(x)) M_t}(y, f(x)-f(x-y))\ge 0$ for all $y\in\bbR$ if also $M_t\le 1$.

Let us consider any $\delta\in(0,\frac 12 M_t]$ and any $x\in\bbR$ such that $|f'(x)|\ge M_t-\delta$.  
We have
\beq\lb{22.29}
h_{\sgn(f'(x)) M_t}(y, f(x)-f(x-y))  \ge \frac{ (1 - M_t^2  - |M_t^{-1}-M_t|\, M_t)\,y^2 }{ y^4 } \ge 0
\eeq
for all $y\in\bbR$ because $|f(x)-f(x-y)|\le M_t|y|$ and $M_t\le 1$.  We also have
 \beq\lb{22.31}
 \left| h_{f'(x)}  (y,  f(x)-f(x-y)) - h_{\sgn(f'(x)) M_t}(y, f(x)-f(x-y)) \right| \le  \frac{  M_t^2+4 }{ M_t y^2 } \,\delta
\eeq
for $y\in\bbR$ because
\[
|(b^{-1}-b) - (a^{-1}-a)| \le |b-a|\, (1+\min\{a,b\}^{-2})
\]
for $a,b>0$.   Writing
\[
f(x)-f(x-y) = f'(x)y - \frac{f''(z)}2y^2 
\]
for some $z$ between $x$ and $x-y$ 
shows that the numerator of $h_{f'(x)}(y, f(x)-f(x-y))$ is
\[
 \frac { (1+f'(x)^2)f''(x)}{2f'(x)} \, y^3 + \frac {(1+f'(x)^2) (f''(z)-f''(x))}{2f'(x)} \,  y^3 - \frac{f''(z)^2}4\, y^4.
\]
Estimating the last two terms by
\[
\frac {2(1+\|f'\|_{C^{1,\gamma}}^3)}{|f'(x)|} |y|^{3+\gamma}
\]
for $|y|\le 1$, and using
\[
\left| \frac {y^3}{(y^2+a_-^2)^2} - \frac {y^3}{(y^2+a_+^2)^2} \right| \le  \frac {2(a_-+a_+)(a_--a_+)}{|y|^3}
\]
with $a_\pm:=f(x)-f(x\pm y)$, we see that there is a constant $C_\gamma$ (depending only on $\gamma$) such that for any $\eps\in[0,1]$ we have
\beq\lb{22.32}
\left| PV \int_{-\eps}^\eps  h_{f'(x)}(y, f(x)-f(x-y))  \, dy \right| 
 \le  \frac {C_\gamma (1+\|f'\|_{C^{1}}^5)}{ |f'(x)|} \,\eps +  \frac {C_\gamma (1+\|f'\|_{C^{1,\gamma}}^3)}{ |f'(x)|} \, \eps^{\gamma}.
\eeq

Finally, we  have
\[
\left| PV \int_{\bbR} \frac{x-y}{(x-y)^2+(f(x)\pm f(y))^2} \, dy \right| \le C_\gamma \|f\|_{C^2_\gamma} 
\]
for some new $C_\gamma$  (see (2.2) in \cite{ZlaMuskatLocal}), as well as
\beq\lb{22.34}
|f''(x)| \le 2 \|f''\|_{\dot C^\gamma}^{1/(1+\gamma)} (M_t - |f'(x)|)^{\gamma/(1+\gamma)} .
\eeq
The latter holds because $|f''(y)-f''(x)|\le \frac12 |f''(x)|$ whenever $|y-x|\le |f''(x)|^{1/\gamma} (2\|f''\|_{\dot C^\gamma} )^{-1/\gamma}$, and so 
\[
M_t - |f'(x)| \ge  \sgn(f'(x)) \left[ f'\left(x+ \sgn(f'(x))\frac {f''(x)^{1/\gamma} } {(2\|f''\|_{\dot C^\gamma})^{1/\gamma}  } \right) - f'(x) \right]
 \ge   \frac {f''(x)^{1/\gamma} } {(2\|f''\|_{\dot C^\gamma})^{1/\gamma} }  \frac {f''(x)}2 .
\]
It follows that (recall also \eqref{2.0aa})
\beq \lb{22.33}
 \left| PV \int_\bbR   \frac{(x-y) f''(x)}{(x-y)^2+(f(x) \pm f(y))^2} \, dy \right| \le 2 C_\gamma \|f\|_{C^{2,\gamma}_\gamma}^{(2+\gamma)/(1+\gamma)} \delta^{\gamma/(1+\gamma)}.
\eeq

Estimates \eqref{22.29}, \eqref{22.31}, \eqref{22.32}, and \eqref{22.33} 
 show that with some new $C_\gamma$
we have 
\beq\lb{22.35}
 \sgn(f_x(x,t))\, \partial_t f_x(x,t)  
 \le  {C_\gamma \left(1+\|f(\cdot,t)\|_{C^{2,\gamma}_\gamma}^5 \right)}  \delta^{\gamma/2}
\eeq
when $M_t>0$, $\delta\in(0,\frac 12 M_t]$, and $|f'(x)|\ge M_t-\delta$
(use $\eps:=\sqrt\delta$ in \eqref{22.32}, and then $|y|\ge \eps$ in \eqref{22.29} and \eqref{22.31}).
Taking $\delta\to 0$ in \eqref{22.35} and using \eqref{1.7} (together with $\gamma\in(0,\frac 12]$) now yields $\frac d{dt} M_t\le 0$.   Moreover, we have
\beq\lb{22.51}
h_{\sgn(f'(x))M_t}(y, f(x)-f(x-y)) \ge \frac{ \left[1 - \left( \frac12 M_t\right)^2  - |M_t^{-1}-M_t|\, \frac 12M_t \right] y^2 }{ y^4 } \ge \frac 1{2y^2}
\eeq
for  $|y|\ge (2\|f\|_{\dot C^{1-\gamma}})^{1/\gamma} M_t^{-1/\gamma}$, since $|  f(x)-f(x-y) | \le    \frac{M_t}2\,|y|$ for these $y$.  Joining this with \eqref{22.29} adds a negative constant to the right-hand side of \eqref{22.35} that is uniform in $\delta$ and $x$,
 which makes the right-hand side of this adjusted version of \eqref{22.35} negative 
 when $\delta$ is small enough.  
 Thus $\frac d{dt} M_t< 0$, which finishes the proof.

\vskip 3mm
\noindent
{\bf Proof of (ii).} 
We proceed as in (i), again assuming that $M_t\in(0,1]$, but this time we have
\[
\begin{split}
\partial_t f'(x) &= 
 \sum_\pm PV \int_\bbR   \frac{(x-y) f''(x)}{(x-y)^2+(f(x) \pm f(y))^2} \, dy  
\\ & -  2{f'(x)} \sum_\pm  PV \int_\bbR  h_{f'(x)}(y, f(x)\pm f(x-y))  \, dy.
\end{split}
\]
Estimates from the proof of (i) apply to three of the integrals above.  As for the fourth, similarly to 
\eqref{22.31} we obtain
 \beq\lb{22.36}
 \left| h_{f'(x)}  (y,  f(x)+f(x-y)) - h_{\sgn(f'(x)) M_t}(y, f(x)+f(x-y)) \right| \le  \frac{  M_t^2+4 }{ M_t^2 (y^2+f(x)^2) } \,\delta
\eeq 
for $y\in\bbR$,
where we used  $|yr|\le y^2+r^2$ instead of $|f(x)-f(x-y)|\le M_t|y|$, as well as that $f(x)+f(x-y)\ge f(x)$ because now $f\ge 0$.  We also have
\[
|f'(x)| \le 2 \|f''\|_{L^\infty}^{1/2} \sqrt{f(x)},
\]
which is proved similarly to \eqref{22.34} (using $f\ge 0$), so
 \beq\lb{22.38}
f(x) \ge  \frac{M_t^2} { 16 \|f''\|_{L^\infty}}
\eeq
when $\delta\in(0,\frac 12 M_t]$, and $|f'(x)|\ge M_t-\delta$.

Hence 
\eqref{22.31}, \eqref{22.32}, \eqref{22.33}, 
\eqref{22.36}, and \eqref{22.38} (but not \eqref{22.29}!) show that 
\beq\lb{22.39}
\begin{split}
 \sgn(f_x(x,t))\, \partial_t f_x(x,t)  
 & \le  {C_\gamma \left(1+\|f(\cdot,t)\|_{C^{2,\gamma}_\gamma}^5 \right)}  \delta^{\gamma/2} 
 + \frac{C_\gamma} {M_t}\sqrt\delta  + \frac {C_\gamma   \|f''(\cdot,t)\|_{L^\infty}^2}{ M_t^4} \sqrt{\delta}
\\ & -  2{f'(x)} \sum_\pm  \int_{|y|\ge\sqrt\delta}  h_{\sgn(f'(x))M_t}(y, f(x)\pm f(x-y))  \, dy
 \end{split}
\eeq
when 
$\delta\in(0,\frac 12 M_t]$ and $|f'(x)|\ge M_t-\delta$, where
$C_\gamma M_t^{-1}\sqrt\delta$ comes from \eqref{22.36} for $|y|\ge\sqrt\delta$, and we also used that \eqref{22.38} and $f\ge 0$ imply
\[
 \int_{-\sqrt\delta}^{\sqrt\delta}  h_{f'(x)}(y, f(x)+ f(x-y))  \, dy \ge -2\sqrt\delta \frac {1+M_t^{-1}}{(M_t^2/16 \|f''\|_{L^\infty})^2}.
 \]

Taking $\delta\to 0$ in \eqref{22.39} 
will conclude $\frac d{dt} M_t<0$ when $M_t\in(0,\frac 3{10}]$ (and hence finish the proof)
if we can prove
\beq\lb{22.40}
\liminf_{\delta\to 0} \, \inf_{f\in \calF_{a,c,\delta}} \, \sum_\pm  \int_{|y|\ge\sqrt\delta}  h_{a}(y, f(0)\pm f(y))  \, dy > 0
\eeq
for any $a\in(0,\frac 3{10}]$ and $c\in[1,\infty)$, where 
\[
\calF_{a,c,\delta} := \left\{f\ge 0\,\big|\, \|f'\|_{L^\infty}= a \,\,\&\,\, f'(0)\le -a+\delta \,\,\&\,\, \|f\|_{C^{2}_\gamma}\le c \right\}.
\]
Here we assume without loss that $x=0$ and $f'(0)<0$, which can be achieved by translating and reflecting $f$, and we also changed variables $y\leftrightarrow -y$, which replaces $h_{-a}$ by $h_a$.

Fix $a,c$ as above, and for each $\delta>0$ and each 
$f\in \calF_{a,c,\delta}$ let
\[
 f_{\delta}(y)=
\begin{cases} 
f(0)-ay  & y\in \left( \frac{f(0)-f(-\sqrt\delta)}a , \frac{f(0)-f(\sqrt\delta)}a \right),
\\  f(\sqrt\delta) & y\in \left( \frac{f(0)-f(\sqrt\delta)}a, \sqrt\delta  \right),
\\  f(-\sqrt\delta) & y\in \left( -\sqrt\delta,\frac{f(0)-f(-\sqrt\delta)}a \right),
\\ f(y) & |y|\ge\sqrt\delta
\end{cases}
\]
(generally $f_{\delta}\notin \calF_{a,c,\delta}$). 
Then $h_{a}(y, f_{\delta}(0)- f_{\delta}(y))=0$ for $y\in (\frac{f(0)-f(-\sqrt\delta)}a , \frac{f(0)-f(\sqrt\delta)}a)$, and 
\beq\lb{22.41}
\sqrt\delta - \frac{c+1}a\delta \le  \pm \frac{f(0)-f(\pm\sqrt\delta)}a  \le\sqrt\delta
\eeq
because
\[
a\ge f'(y)\ge a-\delta-c\sqrt\delta \ge a-(c+1)\sqrt\delta
\]
for $|y|\le\sqrt\delta$.
Hence
\[
\sup_{f\in \calF_{a,c,\delta}}
\left| \int_{|y|\le\sqrt\delta}  h_{a}(y, f_{\delta}(0)- f_{\delta}(y))  \, dy \right| 
\le \frac{2(c+1)}a \delta \sup_{az,w\in[0,(c+1)\delta]} \left|h_a \left( \sqrt\delta-z,a\sqrt\delta - w \right) \right| \to 0
\]
as $\delta\to 0$ because the $\sup$ is $O(\delta^{-1/2})$ (the $O(\delta)$ terms in the numerator of $h_a( \sqrt\delta-z,a\sqrt\delta - w)$ cancel).  Since \eqref{22.38} (with $M_t=a$) shows that for all  $f\in \calF_{a,c,\delta}$ and all  $|y|\le \frac{a^3} { 16 c}$ ($\le a\, f(0)$) we have
\[
|h_{a}(y, f_{\delta}(0)+ f_{\delta}(y))| \le \frac{2 (f(0)+ f_{\delta}(y))^2}{(f(0)+ f_{\delta}(y))^4} \le \frac{2 }{f(0)^2} \le \frac { 512 c^2}{a^4},
\]
 for all small $\delta$ we obtain
\[
\sup_{f\in \calF_{a,c,\delta}}
\left| \int_{|y|\le\sqrt\delta}  h_{a}(y, f_{\delta}(0)+ f_{\delta}(y))  \, dy \right| 
\le \frac { 1024 c^2}{a^4} \sqrt\delta \to 0.
\]
Hence \eqref{22.40} will follow from 
\beq\lb{22.50}
\inf_{f\in \calF_{a,c}} \, \sum_\pm  \int_{\bbR}  h_{a}(y, f(0)\pm f(y))  \, dy > 0
\eeq
for any $a\in(0,\frac 3{10}]$ and $c\in[1,\infty)$, where
\[
 \calF_{a,c}:= \left\{f\ge 0\,\big|\, \|f'\|_{L^\infty} = a = -f'(0)  \,\,\&\,\, \|f\|_{\dot C^{1-\gamma}}\le c \,\,\&\,\, f(0)\ge a^2(16c)^{-1} \right\}
\supseteq \{ f_{\delta}\,|\, f\in \calF_{a,c,\delta} \}
\]

It turns out that the constraints involving $c$ are not needed if we only want to prove \eqref{22.50} with the $\ge$ inequality, which we do next.  We then derive \eqref{22.50} as stated at the end of this section.
Let us from now on fix any $a\in(0,\frac 3{10}]$ and denote
\[
h(y, r):= h_a(y,r)=\frac{y^2 - r^2 - (a^{-1}-a) yr }{ (y^2+r^2)^2 }.
\]

\begin{lemma} \lb{L.2.1}
For $a\in(0,\frac 3{10}]$, let $\calF$ be the set of all $f\ge 0$ such that $f'(0)=-a=\|f'\|_{L^\infty}$.  Then
\beq \lb{22.2}
H(f):= \int_{\bbR}  \big[ h(y, f(0)+f(y)) + h(y, f(0)-f(y)) \big] \, dy \ge 0
\eeq
for all $f\in \calF$, and equality in \eqref{22.2} holds only when $f(y)\equiv a|y- {f(0)}a^{-1}|$.
\end{lemma}

{\it Remark.}
Note that there is no PV needed in \eqref{22.2} for $f\in\calF$.  
Indeed,  $h(y, f(0)+f(y))$ is  integrable on $\bbR$ because $f(0)>0$, while with
\[
\alpha(y):=a-\frac{f(0)-f(y)}y  \in [0,2a]
\]
we have
\beq \lb{22.1b}
h(y, f(0)-f(y)) \ge \frac{a^{-1}-a} {[1+(a-\alpha(y))^2]^2} \, \frac  {\alpha(y)} {y^2} \ge 0
\eeq
because $a\in(0,1]$, so $\int_\bbR h(y, f(0)-f(y))\,dy\in[0,\infty]$ also exists.
\smallskip

\begin{proof}
The remark above shows that the integral in \eqref{22.2}  exists and belongs to $\bbR\cup\{\infty\}$.
Next, if $f\in\calF$ and $\tilde f(y):=\frac{f(f(0)y)}{f(0)}$, then $\tilde f\in\calF$ and
$H(\tilde f)=f(0) H(f)$.
Therefore it suffices to consider the case $f(0)=1$, which we do from now on.   

Then  for any $S \subseteq \bbR$ we have
\beq \lb{22.2c}
\int_{S}   h(y, 1+f(y))  \, dy \ge -2\int_0^\infty \frac{y^2 + (2+ay)^2 + a^{-1} y(2+ay) }{ \max\{y,2-ay\}^4 } \, dy 
=: I \in(-\infty,0),
\eeq
which together with \eqref{22.1b} imples $H(f)\ge I$.
Since we also have
\beq \lb{22.2a}
H(a|y- a^{-1}|)= \int_\bbR  \big[ h(y, 2-ay) + h(y, ay) \big] \, dy =0
\eeq
because $ h(y, ay)=0$ for all $y\in\bbR$  and 
\beq \lb{22.2b}
a(1+a^2)\, h(y, 2-ay) = \frac d{dy} \, \frac { 1-3a^2 -2a(1-a^2)y  } { y^2+(2-ay)^2 },
\eeq
it suffices to consider only $f$ with $H(f)\le 1$, which then implies
\[
1-I\ge \int_{\bbR}   h(y, 1-f(y))  \, dy =  \int_{\bbR} \frac{a^{-1}-a} {[1+(a-\alpha(y))^2]^2} \, \frac  {\alpha(y)} {y^2}\, dy.
\]

If for some $y_0\in(0,1]$ we have $y_0\alpha(y_0)\ge Cy_0^{3/2}$ with $C:=1-I$, then from $(y\alpha(y))'=a+f'(y)\in[0,2a]$ we obtain
\[
y\alpha(y)\ge (C-2a) y_0^{3/2} \ge (C-2a) y^{3/2}
\]
for $y\in [y_0- y_0^{3/2},y_0]$, and hence \eqref{22.1b} shows that
\[
 \int_\bbR   h(y, 1-f(y))  \, dy \ge   \int_{y_0-y_0^{3/2}}^{y_0} \frac{a^{-1}-a} {[1+a^2]^2} \, \frac{C-2a}{y^{3/2}}\, dy 
 \ge \frac{(a^{-1}-a)(1-2a)} {[1+a^2]^2} (1-I) > 1-I
\]
because $a\in(0,\frac 3{10}]$ and $I\le 0$.  The same argument applies to $y_0\in[-1,0)$, so we see that $H(f)\le 1$ forces $|y|\alpha(y)\le C|y|^{3/2}$ and so
\beq \lb{22.2d}
0\le \sgn(y)\,\big( f(y)-(1-ay) \big) \le C|y|^{3/2}
\eeq
for all $y\in [-1,1]$.  We will make use of this estimate in the next section.

Next, \eqref{22.2b} shows that it suffices to show 
\beq \lb{22.3}
H(f) > H(a|y- a^{-1}|)
\eeq
when $f$ is not $a|y- a^{-1}|$ (recall that we assume $f\in\calF$ and $f(0)=1$).
Letting $A:={a^{-1}}-a$, we have
\[
 \partial_r h(y,r) = \frac{2r^3+3Ayr^2 -6y^2r -Ay^3 }{ (y^2+r^2)^3 } =\frac 1{y^3} \, g \left(\frac ry \right) ,
\]
where
\beq \lb{22.3a}
g(s):= \frac{2s^3+3As^2 -6s -A }{ (1+s^2)^3 }.
\eeq

We will separately study the integral in \eqref{22.2} over $y<0$ and $y>0$.
We start with the former, which will be much simpler to minimize.  This is because at the end of this proof we will show that
\beq \lb{22.4}
  \int_{s_0-a}^{s_0-b} g(s) ds >\int_b^a g(s)ds
\eeq
for any $b\in [-a,a)$ and $s_0< 0$.  This now implies that if $y<0$ and we let $b:=\frac {1-f(y)}{y} \in [-a,a]$ and $s_0:=\frac{2}y<0$, then
\begin{align*}
y^2 & \big[ h(y,1+f(y)) - h(y,2-ay) + h(y,1-f(y)) - h(y,ay) \big]
\\ & \qquad \qquad = \frac 1y \int_{2-ay}^{1+f(y)} g\left( \frac ry \right) dr 
+ \frac 1y \int_{ay}^{1-f(y)} g\left( \frac ry \right) dr
\\ & \qquad \qquad =  \int_{s_0-a}^{s_0-b} g(s) ds 
+  \int_{a}^{b} g(s) ds \ge 0,
\end{align*}
with equality   at the end only when $f(y)=1-ay$ (i.e., $b=a$).  This  shows that $1-ay$  is the  unique minimizer of
\[
H^-(f):=  \int_{-\infty}^{0}  \big[ h(y, 1+f(y)) + h(y, 1-f(y)) \big] \, dy
\]
 among all Lipschitz $f:(-\infty,0]\to[0,\infty)$ such that $f'(0)=-a=\|f'\|_{L^\infty}$ and $f(0)=1$.  
 Recalling \eqref{22.3}, it now suffices to show that $a|y- a^{-1}|$ is the unique minimizer of
\beq \lb{22.5}
H^+(f):=  \int_0^\infty  \big[ h(y, 1+f(y)) + h(y, 1-f(y)) \big] \, dy
\eeq
on the set of all Lipschitz $f:[0,\infty)\to[0,\infty)$ such that $f'(0)=-a=\|f'\|_{L^\infty}$ and $f(0)=1$.  This is the claim of the next lemma, which we prove in the next section, so the present proof will be finished once we establish \eqref{22.4}.

To this end we note that from
\[
g'(s)= -6\, \frac {s^4+2As^3-6s^2-2As+1} {(1+s^2)^4}
\]
it follows that $g$ has at most 4 local extrema and $\sgn(s)g(s)>0$ when $|s|$ is large enough.
Since $a\in(0, \frac 3{10}]$ also implies $a^{-1}>10a$ and so
\begin{align*}
g(-a^{-1}) & =(a^{-3}+2a^{-1}+a)(1+a^{-2})^{-3} \,\,\, >0,
\\ g(-a) & = -(a^{-1}-10a+5a^3)(1+a^2)^{-3} \,\,\, <0,
\\ g(0) & =-A \,\,\, < g(-a),
\\ g(a) & = - (a^{-1}+2a +a^{3} )(1+a^{2})^{-3} \,\,\, <  g(0),
\\ g(a^{-1}) & = (5a^{-3}-10a^{-1}+a)(1+a^{-2})^{-3} \,\,\, >0,
\\ g'(a) & = 6\,\frac{1-2a+4a^2+a^4}{(1+a^2)^2} \,\,\, >0,
\end{align*}
we see that $g$ does have 4 local extrema at
\beq \lb{22.6}
s_1\in(-\infty,-a^{-1}), \qquad s_2\in(-a^{-1},-a), \qquad s_3\in(0,a), \qquad s_4\in(a,\infty),
\eeq
and these are negative local minima at $s_1,s_3$ and positive local maxima at $s_2,s_4$.  Figure \ref{F.2.3} shows $g$ when $a=\frac 3{10}$, and hence $A=3$.

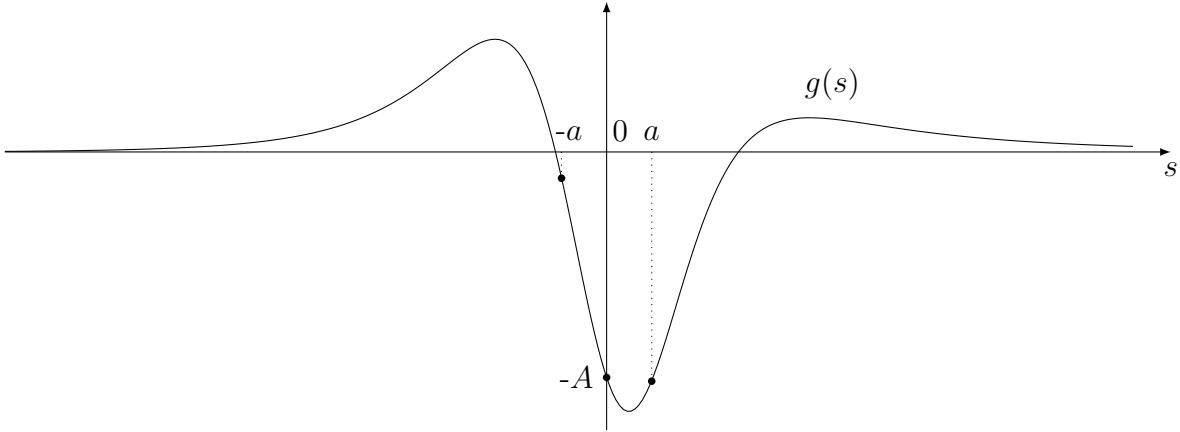
\begin{figure}
\begin{tikzpicture}
  \draw[->,>=latex] (0,0) -- (15.5,0) node[below]{$s$} ;
  \draw[->,>=latex] (8,-3.7) -- (8,2)  ;
  \node (x) at (11,0.9) {$g(s)$};
  \draw[dotted] (8.6,0) -- (8.6,-3.05);
  \node (x) at (8.6,0.25) {$a$};
  \draw[dotted] (7.4,0) -- (7.4,-0.35);
  \node (x) at (7.5,0.25) {-$a$};
    \node (x) at (8.18,0.28) {$0$};
  \node (x) at (7.6,-3) {-$A$};
  \draw [domain=0:15,samples=300] plot (\x, {(2*((\x-8)/2)^3+9*((\x-8)/2)^2 -6*((\x-8)/2) -3)/ (1+((\x-8)/2)^2)^3} ) ;
  \fill (7.4,-0.35) circle(0.05);
  \fill (8.6,-3.05) circle(0.05);
   \fill (8,-3) circle(0.05);
\end{tikzpicture}
\caption{The function $g$ for $a=\frac 3{10}$ (and so $A=3$).} \label{F.2.3}
\end{figure}

We next note that 
\[
g(s_1) \ge  \frac{2s_1^3 -A }{ (1+s_1^2)^3 } \ge 3 \left( \frac{s_1 }{ 1+s_1^2 } \right)^3 \ge   3 \left( \frac{-a^{-1} }{ 1+a^{-2} } \right)^3  =  - \frac{3a^3 }{ (1+a^2)^3 } > g(-a),
\]
so $g(s)> g(-a)$ for all $s<-a$.  This and $g(s)<g(-a)$ for $s\in(-a,a]$ show that
\[
\int_{s_0-a}^{s_0-b} g(s) ds > (a-b) g(-a) > \int_{-a}^{-b} g(s)ds
\]
when $b\in[-a,a)$ and $s_0-b\le -a$.  We also have
\[
\frac d{ds_0}\,\int_{s_0-a}^{s_0-b} g(s) ds = g(s_0-b)-g(s_0-a) < 0
\]
when $s_0-b\in[-a,a-(b)_+]$ (i.e., $s_0\in[b-a,a-(b)_-]$) because $g(s)> g(-a)$ for all $s<-a$,  $g$ is decreasing on $[-a,0]$, and $g(s)<g(0)$ for $s\in(0,a]$.  Hence if $b\le 0$, then for all $s_0<a+b$ ($=a-(b)_-$) we have
\[
\int_{s_0-a}^{s_0-b} g(s) ds > \int_{b}^{a} g(s) ds,
\]
while if $b\ge 0$, then for all $s_0<b$ ($\le a-(b)_-$) we have
\[
\int_{s_0-a}^{s_0-b} g(s) ds > \int_{b-a}^{0} g(s) ds > (a-b)g(0) >  \int_{b}^{a} g(s) ds.
\]
In both cases, the range of $s_0$ for which the inequality holds includes all $s_0<0$, therefore we proved \eqref{22.4}.
\end{proof}

Let us next show \eqref{22.50}. Assume that $f_n\in\calF_{a,c}$ is a sequence such that $\lim_{n\to\infty} H(f_n)=0$ (and  $H(f_n)\ge 0$ for all $n$ by Lemma \ref{L.2.1}).  If $\limsup_{n\to\infty} f_n(0) = \infty$, then clearly
\[
\limsup_{n\to\infty} \int_0^\infty \ h(y, f_n(0)+f_n(y))  \, dy \ge 0,
\]
while \eqref{22.29} and \eqref{22.51} show that
\[
\inf_{n} \int_0^\infty \ h(y, f_n(0)-f_n(y))  \, dy >0.
\]
Hence $\limsup_{n\to\infty} H(f_n)>0$, a contradiction, so we have $m:=\sup_n f_n(0)<\infty $.  Since also $f_n(0)\ge \frac{a^2}{16c}$, we see that
\[
\tilde f_n(y):=\frac{f_n(f_n(0)y)}{f_n(0)} \in \calF_{a,c'}
\]
with $c':=c(\frac {a^2}{16c})^{-\gamma}$.  Now $H(\tilde f_n)=f_n(0)H(f_n)\le m \, H(f_n)\to 0$ and $\tilde f_n(0)=1$.   So the argument from the start of the proof of the next lemma, applied to $H$ and on $\bbR$ instead to $H^+$ on $\bbR^+$, shows that some subsequence $\{f_{n_k}\}$ converges to some $f\in \calF_{a,c'}$ with $f(0)=1$ such that $H(f)=0$.  But then Lemma \ref{L.2.1} shows that $f(y)\equiv a|y- a^{-1}|\notin \calF_{a,c'}$, a contradiction.

It now remains to prove the following result, which was used in the proof of Lemma \ref{L.2.1}.  We do this in the next section.

\begin{lemma} \lb{L.2.2}
For   $a\in(0,\frac 3{10}]$, let $\calF^+$ be the set of all  Lipschitz $f:[0,\infty)\to[0,\infty)$ such that $f'(0)=-a=\|f'\|_{L^\infty}$ and $f(0)=1$.  Then $a|y- a^{-1}|$ is the unique minimizer of $H^+$ 
on  $\calF^+$.
\end{lemma}


\section{Proof of Lemma \ref{L.2.2}} \lb{S3}


From \eqref{22.2c} and \eqref{22.1b} we again see that 
\[
J:=\inf_{f\in\calF^+} H^+(f)\ge I
\]
and \eqref{22.2d} holds for all $y\in[0,1]$ whenever $H^+(f)\le J+1$.
Let $f_n\in\calF^+$ be such that  $H^+(f_n)\le J+\frac 1n$.  The Arzel\` a-Ascoli Theorem and \eqref{22.2d} 
show that there is a subsequence $\{f_{n_k}\}$ that converges  to some $f\in\calF^+$ in $L^\infty_{\rm loc}([0,\infty))$.
Since for all $\tilde f\in\calF^+$ we have
\[
 \int_M^\infty  \big| h(y, 1+\tilde f(y)) + h(y, 1-\tilde f(y)) \big| \, dy \le \int_M^\infty \frac{y^2 + (2+ay)^2 + a^{-1} y(2+ay) }{ y^4 } \, dy \to 0
\]
as $M\to\infty$, it follows that for each $\sigma\in(0,a)$ we have
\[
I_{n_k,\sigma}:= \int_\sigma^\infty   \big[ h(y, 1+ f_{n_k}(y)) + h(y, 1- f_{n_k}(y)) \big]  dy \to \int_\sigma^\infty \big[  h(y, 1+ f(y)) + h(y, 1- f(y)) \big] dy
 \]
 as $k\to\infty$.  Since $I_{n_k,\sigma}\le H^+(f_{n_k})+7\sigma\le J+\frac 1{n_k}+7\sigma$ for all $k$, which holds by \eqref{22.1b} and $h(y, 1+ f_n(y))\ge -7$ for all $(y,n)\in[0,a]\times\bbN$, we see that
 \[
 \int_\sigma^\infty  \big[ h(y, 1+ f(y)) + h(y, 1- f(y))  \big] dy \le J +7\sigma
 \]
 for each $\sigma\in(0,a)$.  Hence $H^+(f)\le J$, which means that $H^+$ has a minimizer in $\calF^+$.

Consider now any such minimizer $f$, and for $0< p<q$ and $\delta>0$ let
\[
f_{p,q,\delta}^+(s):=
\begin{cases}
f(s) & s\notin(p,q), \\
\min \{ f(s)+\delta, f(p)+a(s-p), f(q)+a (q-s)  \} & s\in(p,q)
\end{cases}
\]
and 
\[
f_{p,q,\delta}^-(s):=
\begin{cases}
f(s) & s\notin(p,q), \\
\max \{ f(s)-\delta, f(p)-a(s-p), f(q)-a (q-s)  \} & s\in(p,q).
\end{cases}
\]
Then clearly $f_{p,q,\delta}^+\in\calF^+$, while $f_{p,q,\delta}^-\in\calF^+$ as long as $f\ge\delta$ on $[p,q]$, and we let
(with every statement below that includes $\pm$ being two separate statements)
\begin{align*}
p_{0^+}^\pm & := \inf \big\{s\in(p,q)\,\big|\, f'\not\equiv \pm a \text{ on $(p,s)$} \big\},
\\ p_{\delta^+}^\pm & := \inf \left\{ s\in(p,q) \,\bigg|\, \int_p^s (a\mp f'(s)) ds =\delta \right\} \quad \searrow p_{0^+}^\pm \text{ as $\delta\to 0^+$},
\\ q_{0^-}^\pm & := \sup \big\{s\in(p,q)\,\big|\, f'\not\equiv \mp a \text{ on $(s,q)$} \big\},
\\ q_{\delta^-}^\pm & := \sup \left\{ s\in(p,q) \,\bigg|\, \int_s^q (a\pm f'(s))ds =\delta \right\}  \quad \nearrow q_{0^-}^\pm \text{ as $\delta\to 0^+$}
\end{align*}
(with $\inf\emptyset :=\infty$ and $\sup\emptyset :=-\infty$). Then for each $\delta>0$ we have 
\begin{align*}
f_{p,q,\delta}^\pm & \equiv f \qquad \qquad \text{on $(-\infty,p_{0^+}^\pm]\cup [q_{0^-}^\pm,\infty)$}, 
\\ f_{p,q,\delta}^\pm & \equiv f\pm\delta \,\, \qquad
\text{on $[p_{\delta^+}^\pm,q_{\delta^-}^\pm]$ if $p_{\delta^+}^\pm\le q_{\delta^-}^\pm$}. 
\end{align*}
Since $f$ minimizes $H^+$, it follows that whenever $p_{0^+}^\pm < q_{0^-}^\pm$ (as well as $\min_{s\in[p,q]} f(s)>0$ in the ``$-$'' case), we have
\begin{align*}
0\le \lim_{\delta\to 0^+} \frac {H^+(f_{p,q,\delta}^\pm) - H^+(f)}\delta 
& = \pm \int_{p_{0^+}^\pm}^{q_{0^-}^\pm} \big[ h_r(y, 1+f(y)) - h_r(y, 1-f(y)) \big] \, dy
\\ & =\pm \int_{p_{0^+}^\pm}^{q_{0^-}^\pm} \frac 1{y^3} \left[ g \left(\frac{1+f(y)} y \right) - g \left(\frac{1- f(y)} y \right) \right] \, dy.
\end{align*}
This shows that if $f'\not\equiv \pm a$ on $(p,p+\eps)$ and $f'\not\equiv \mp a$ on $(q-\eps,q)$ for all $\eps>0$, and also $\min_{s\in[p,q]} f(s)>0$ in the ``lower sign'' case, then
\beq \lb{3.1}
\pm \int_{p}^{q} \frac 1{y^3} \left[ g \left(\frac{1+f(y)} y \right) - g \left(\frac{1- f(y)} y \right) \right] \, dy \ge 0 .
\eeq

Assume that $f'\not\equiv a$ and $f'\not\equiv -a$ on $(p,p+\eps)$  for all $\eps>0$.  Then there are $q_n\searrow p$ such that $f'\not\equiv - a$ on $(q_n-\eps,q_n)$ for all $\eps>0$, so \eqref{3.1} yields
\[
0\le \frac{p^3}{q_n-p} \, \int_{p}^{q_n} \frac 1{y^3} \left[ g \left(\frac{1+f(y)} y \right) - g \left(\frac{1- f(y)} y \right) \right] \, dy
 \to g \left(\frac{1+f(p)} p \right) - g \left(\frac{1- f(p)} p \right).
\]
There are also $q_n'\searrow p$ such that $f'\not\equiv  a$ on $(q_n'-\eps,q_n')$ for all $\eps>0$, so \eqref{3.1}  again yields
\[
0\ge \frac{p^3}{q_n'-p} \, \int_{p}^{q_n'} \frac 1{y^3} \left[ g \left(\frac{1+f(y)} y \right) - g \left(\frac{1- f(y)} y \right) \right] \, dy
 \to g \left(\frac{1+f(p)} p \right) - g \left(\frac{1- f(p)} p \right),
\]
provided $f(p)>0$.
This will yield part (i) of the following result.

\begin{lemma} \lb{L.3.1}
Let $f$ be a minimizer of $H^+$ on $\calF^+$ and $p\in(0,\infty)$.

(i) 
If $f'\not\equiv a$ and $f'\not\equiv -a$ on $(p,p+\eps)$  for all $\eps>0$, then
\beq \lb{3.2}
g \left(\frac{1+f(p)} p \right) = g \left(\frac{1- f(p)} p \right).
\eeq

(ii)
If $f'\not\equiv a$ and $f'\not\equiv -a$ on $(p,p-\eps)$  for all $\eps>0$, then \eqref{3.2} holds.  

(iii)
If $f'\equiv - a$ on $(p-\eps,p)$ and $f'\equiv a$ on $(p,p+\eps)$  for some $\eps>0$, then
\beq \lb{3.3}
g \left(\frac{1+f(p)} p \right) \ge g \left(\frac{1- f(p)} p \right).
\eeq

(iv)
If $f'\equiv a$ on $(p-\eps,p)$ and $f'\equiv -a$ on $(p,p+\eps)$  for some $\eps>0$, then
\[
g \left(\frac{1+f(p)} p \right) \le g \left(\frac{1- f(p)} p \right).
\]
\end{lemma}

\begin{proof}
We proved (i) when $f(p)>0$, but it obviously also holds when $f(p)=0$.  We can obtain (ii) via a similar argument (with $q_n\nearrow p$).  In (iii) we see that \eqref{3.1} holds with $\pm$ being $+$ and with $(p-\eps,p+\eps)$ in place of $(p,q)$ whenever $\eps>0$ is small enough, so multiplying it by $\frac{p^3}{2\eps}$ and taking $\eps\to 0$ yields \eqref{3.3}.
Finally, (iv) is obtained similarly when $f(p)>0$, via \eqref{3.1} with $\pm$ being $-$, while it holds trivially when $f(p)=0$.
\end{proof}

Consider the open sets
\[
P_\pm:= \left\{ y>0 \, \bigg|\, \pm g \left(\frac{1+f(y)} y \right) > \pm g \left(\frac{1- f(y)} y \right) \right\} 
\]
and the set
\[
P_0:= \left\{ y>0 \, \bigg|\, g \left(\frac{1+f(y)} y \right) = g \left(\frac{1- f(y)} y \right) \right\} .
\]
Lemma \ref{L.3.1} is now clearly equivalent to 
(i) and (ii) in 
the following result.  In it we also let $s^\pm_y:= \frac{1\pm f(y)} y$, which satisfy
\beq \lb{3.4a}
-a\le s^-_y\le \min \left\{a,\frac 1y \right\} \le \frac 1y \le s^+_y.
\eeq

\begin{lemma} \lb{L.3.2}
(i)
For each $y\in P_+$ there is $\eps>0$ such that either $f'\equiv -a$  on $(y-\eps,y+\eps)$,  or $f'\equiv a$ on $(y-\eps,y+\eps)$, or $f'\equiv -a$  on $(y-\eps,y)$ and $f'\equiv a$  on $(y,y+\eps)$.

(ii)
For each $y\in P_-$ there is $\eps>0$ such that either $f'\equiv a$  on $(y-\eps,y+\eps)$,  or $f'\equiv -a$ on $(y-\eps,y+\eps)$, or $f'\equiv a$  on $(y-\eps,y)$ and $f'\equiv -a$  on $(y,y+\eps)$.

(iii)
If $\frac{1+f(y_0)}{y_0}\le a$ for some $y_0\in P_0$, then with $y_1:=\inf \{y\ge y_0\,|\, f(y)=0\}$ we have $(y_0,y_1)\in P_-$.
\end{lemma}

\begin{proof}
It remains to prove (iii), so assume that $y_0$ is as above, and also that $f(y_0)>0$ (otherwise the result is trivial).  First note that $s^+_{y_0}\le a$ implies $s^+_{y}\le a$ for all $y\ge y_0$.  Then $y_0\in P_0$ and $f(y_0)>0$ show that if $s'\in(0,s_3)$ is the unique number with $g(s')=g(a)$ (recall that $g(0)>g(a)$), then
\[
0< s'\le s^-_{y_0} <s_3< s^+_{y_0}\le a
\]
(and in particular, $g'(s^-_{y_0})<0<g'(s^+_{y_0})$).  
We have
\beq \lb{3.4c}
y \, \frac d{dy} \left( g(s^+_y) - g(s^-_y) \right) = - \left(s^+_{y}  g' ( s^+_{y} ) + s^-_{y}  |g' ( s^-_{y} )| \right) + f'(y) \left( g' ( s^+_{y} )+g' ( s^-_{y} ) \right) 
\eeq
for all $y>0$, and we also claim that
\beq\lb{3.4b}
r  g' ( r ) + q  |g' ( q )| > a \left| g' ( r )+g' ( q ) \right|
\eeq
whenever $s'\le q<r\le a$ and $g(r)=g(q)$.
From this and \eqref{3.4c}, and $g ( s^+_{y_0} )=g ( s^-_{y_0} )$ we see that $\frac d{dy} \left( g(s^+_y) - g(s^-_y) \right)<0$ holds for $y=y_0$ as well as for any other $y\in[y_0,y_1)\cap P_0$.  Therefore we must have $(y_0,y_1)\in P_-$, and it remains to prove \eqref{3.4b} assuming that $s'\le q<r\le a$ and $g(r)=g(q)$.
%

If $g''$ were constant on $[0,a]$, we would have $g' ( r )=-g' ( q )\neq 0$ and \eqref{3.4b} would follow trivially.  We therefore want to show that $g''$ does not vary too much on $[0,a]$.  We note that this can be done easily for all small enough $a>0$, but we provide here a quantitative argument that applies to all $a\in(0,\frac3{10}]$. Since
\begin{align*}
\frac{g''(s)}{12} & = \frac { A+10s-10As^2 -20s^3 +5As^4+2s^5}{(1+s^2)^5},
\\ \frac{g'''(s)}{120} & = \frac { 1-3As-15s^2 +10As^3+15s^4-3As^5 -s^6}{(1+s^2)^6},
\end{align*}
and $a\le\frac 3{10}$, we see that $g''$ increases on $[0,s'']$ and decreases on $[s''',a]$ for some $s''<s'''$ and both close to $\frac a3$ (because the second numerator above is approximately $1-3a^{-1}s$).  In fact, a simple analysis of the three parentheses in
\begin{align*}
3as -15s^2 & +10As^3+15s^4-3As^5 -s^6 
\\ &= (3as-15s^2+10a^{-1}s^3)  -( 10as-15s^2 +3a^{-1}s^3)  s^2 + (3as-s^2)s^4
\end{align*}
shows that when $s\in[0,a]$, they take values in $[-2a^2,\frac {a^2}5]$, $[-2a^2,2a^2]$, and $[0,2a^2]$, respectively.  Thus the absolute value of the whole expression is bounded by $2a^2(1+a^2+a^4) \le \frac{11a^2}5$, which
 means that for  $s\in[0,a]$ we have $\sgn(s-\frac a3)g'''(s)\ge 0$ when $ 3a^{-1} |s-\frac a3|\ge \frac{11a^2}5$.  This holds when $|s-\frac a3|\ge \frac a{15}$ because $a^2\le\frac 9{100}$, so we can pick $s'':=\frac{a}3-\frac a{15}$ and $s''':=\frac{a}3+\frac a{15}$ above.

Then $|g'''(s)|\le 120(3a^{-1}  \frac a{15} + \frac{11a^2}5)\le 48$ on $[s'',s''']$, so $|g''(s)-g''(\frac a3)| \le \frac {16a}5\le 1$ there.  Hence $g''(s)$ takes values in $[g''(0), g''(\frac a3)+1]$ on $[0,s_3]$ and in $[g''(a),g''(\frac a3)+1]$ on $[s_3,a]$.  

A simple estimate thus yields
\[
7.5A-1\le 12 \left( A-\frac 3{20} \right) (1+a^2)^{-5} \le g''(s) \le 12(A+0.7)+1 \le  12A+10
\]
for $s\in[0,a]$.  This shows that if $g(r)=g(q)$ for some  $s'\le q<r\le a$, then 
\[
\max\left\{  \frac{|q-s_3|}{|r-s_3|} ,   \frac{|g'(q)|}{|g'(r)|} \right\}  \le \sqrt{\frac {12A+10}{7.5A-1}} \le \frac 32
\]
because $A\ge 3$.  Then $s_3\ge s'+\frac {2}5 (a-s')\ge \frac {2a}5$ and
\[
r  g' ( r ) + q  |g' ( q )| \ge \frac{2ag' ( r )}5 > \frac{ag' ( r )}3 \ge a \left| g' ( r )+g' ( q ) \right|.
\]
Hence we proved \eqref{3.4b} and the proof of (iii) is finished.
\end{proof}

From \eqref{3.4a} we see that
$g(s^-_y)\le g(-a)<0<g(s^+_y)$ holds for all small enough $y>0$  (recall that $g(a)<g(-a)$) and so $(0,\sigma)\subseteq P_+$ for some $\sigma>0$.  Let us pick the largest such $\sigma$.  Since also $f'(0)=-a$, it follows from Lemma \ref{L.3.2}(i) that $f(y)=1-ay$ on $[0,p]$ for some $p\in(0,{a^{-1}}]$ (again pick the largest such $p$).  We first claim that $p={a^{-1}}$.

Assume therefore that $p<{a^{-1}}$.  Then $s^-_y=a<s^+_y$ for all $y\le p$, so $\sigma>p$.  Lemma~\ref{L.3.2}(i) then shows that $f(y)=1+a(y-2p)$ on $[p,\sigma]$.  Let $q\ge \sigma$ be the largest number such that $f(y)=1+a(y-2p)$ on $[p,q]$ (or let $q:=\infty$ if $f(y)=1+a(y-2p)$ on $[p,\infty)$).  For all $\delta\in(0,{a^{-1}}-p]$ let
\[
\calF^+\ni f_\delta(y):=
\begin{cases}
1-ay & y\in[0,p+\delta],
\\ \min \{ 1+a(y-2(p+\delta)), f(y)\}  & y> p+\delta.
\end{cases}
\]
Then $f_\delta=f-2a\delta$ on $[p+\delta,q)$ and $f_\delta=f$ on $[q_\delta,\infty)$ for some $q_\delta \searrow q$ as $\delta\to 0$ (the latter if $q<\infty$).  It follows that
\beq\lb{3.5}
\begin{split}
0\le \lim_{\delta\to 0^+} \frac {H^+(f_{\delta}) - H^+(f)}{2a\delta} 
& = \int_{p}^{q} \big[ -h_r(y, 1+f(y)) + h_r(y, 1-f(y)) \big] \, dy
\\ & = \int_{p}^{q}  \frac{ -g ({(2-2ap)y^{-1}+a} ) + g (2apy^{-1}-a) } {y^{3}} \, dy.
\end{split}
\eeq
Arguments of both functions in the last integral are strictly decreasing in $y$, and for $y\in[p,\infty)$ we have
\[
-a<2apy^{-1}-a \le a < (2-2ap)y^{-1}+a.
\]
The properties of $g$ then imply that there is $q_0\in[p,\infty)$ such that the integrand is negative on $[p,q_0)$ and positive on $(q_0,\infty)$.  If we therefore show that
\beq \lb{3.6}
\int_{p}^{\infty}  \frac{ -g ({(2-2ap)y^{-1}+a} ) + g (2apy^{-1}-a) } {y^{3}} \, dy <0,
\eeq
it will follow that this integral with any $q\in(p,\infty]$ in place of $\infty$ is also negative, which will contradict  \eqref{3.5}.

%
%

One can directly integrate in \eqref{3.6}, but it will be simpler to instead  keep the limit in \eqref{3.5} outside of the integral.  We see that the left-hand side of \eqref{3.6} is
\[
\lim_{\delta\to 0^+} \frac {K_{\delta} - K_0}{2a\delta} ,
\]
where
\[
K_\delta:= \int_{p}^{\infty} \big[ h(y, 2+a(y-2p-2\delta)) + h(y, a(2p+2\delta-y)) \big] \, dy.
\]
We have
\beq \lb{3.7}
h(y,B\pm ay)=   
\frac{y^2 - (B \pm ay)^2 - A y(B \pm ay) }{ [y^2+(B \pm ay)^2]^2 } 
= \partial_y G^\pm(B,y),
\eeq
where (recall that $A=\frac {1-a^2}a$)
\begin{align*}
G^+(B,y) &:=
 \frac 1{2a}\, \frac {B}{y^2+( B + ay)^2 },
\\  G^-(B,y) &:=
 \frac{1-3a^2}{2a(1+a^2)} \, \frac {B}{y^2+( B - ay)^2 } -  \frac{2(1-a^2)}{1+a^2}  \, \frac {y} {y^2+( B - ay)^2 } ,
\end{align*}
and then
\begin{align*}
\partial_B G^+(B,y) &=   \frac 1{2a}\, \frac {(1+a^2)y^2-B^2}{[y^2+( B + ay)^2]^2 } ,
\\  \partial_B G^-(B,y) &=   \frac 1{2a(1+a^2)}\, \frac {(3a^2-1)B^2 + 8a(1-a^2)yB+(1-10a^2+5a^4)y^2}{[y^2+( B - ay)^2]^2 } .
\end{align*}
This and $p\in (0,{a^{-1}})$ means that if we first integrate in $y$ and then take $\delta\to 0^+$, we obtain
\begin{align*}
\lim_{\delta\to 0^+} \frac {K_{\delta} - K_0}{\delta} 
 & = 2a \, \partial_B G^+( 2-2ap,p)  - 2a \, \partial_B G^-( 2ap,p) 
 \\ & =  \frac {(1+a^2)p^2-4(1-ap)^2}{[p^2+( 2-ap)^2]^2 } - \frac {(1+a^2)p^2}{[p^2+( ap)^2]^2 } 
 < - \frac {4(1-ap)^2}{[p^2+( 2-ap)^2]^2 }<0 ,
\end{align*}
proving \eqref{3.6}.  This yields a contradiction, so indeed $p={a^{-1}}$ and $f(y)=1-ay$ on $[0,{a^{-1}}]$.

Next, let us consider $f$ immediately to the right of $y={a^{-1}}$.  One option is $f\equiv 0$ on $[{a^{-1}},{a^{-1}}+\eps]$ for some $\eps>0$, and we analyze it later.  Otherwise we have ${a^{-1}}\in\partial P$, where  $P:=\{y>{a^{-1}} \,|\, f(y)>0\}$.  Since $(1\pm f({a^{-1}}))/{a^{-1}}=a$ and $g'(a)>0$, it follows that $P\cap({a^{-1}},{a^{-1}}+\eps)\subseteq P_+$ for some $\eps>0$.  So Lemma \ref{L.3.2}(i) applies on this set and we see that $f(y)=a(y-{a^{-1}})=ay-1$ on $({a^{-1}},{a^{-1}}+\eps)$.

If $f(y)=ay-1$ does not hold on all of $[a^{-1},\infty)$, let $q\in[a^{-1}+\eps,\infty)$ be the largest number such that this equality holds on $[a^{-1},q]$.  We can now proceed similarly to the case $p<a^{-1}$ above to obtain contradiction with $f$ being a minimizer of $H^+$ on $\calF^+$.  If
\[
f_\delta(y):=
\begin{cases}
\max\{1-ay,0\} & y\in[0,a^{-1}+\delta],
\\ \min \{ ay-1-a\delta, f(y)\}  & y> a^{-1}+\delta,
\end{cases}
\]
then
\beq\lb{3.8}
0\le \lim_{\delta\to 0^+} \frac {H^+(f_{\delta}) - H^+(f)}{a\delta} 
 = \int_{a^{-1}}^{q}  \frac{ -g ({a} ) + g (2y^{-1}-a) } {y^{3}} \, dy.
\eeq
The integrand is again negative on $(a^{-1},q_0)$ and positive on $(q_0,\infty)$ for some $q_0\in(a^{-1},\infty)$, so  it suffices to show that
\beq\lb{3.9}
\int_{a^{-1}}^{\infty}  \frac{ -g ({a} ) + g (2y^{-1}-a) } {y^{3}} \, dy \le 0
\eeq
in order to get a contradiction with \eqref{3.8} (since also $q<\infty$).  Letting
\[
K_\delta:= \int_{a^{-1}}^{\infty} \big[ h(y, ay-a\delta) + h(y, 2+a\delta-ay)) \big] \, dy,
\]
we again see that the left-hand side of \eqref{3.9} equals 
\begin{align*}
\lim_{\delta\to 0^+} \frac {K_{\delta} - K_0}{a\delta} 
  =  \partial_B G^+( 0,a^{-1})  -  \partial_B G^-( 2,a^{-1}) = \frac a{2(1+a^2)} - \frac a{2(1+a^2)} =0.
\end{align*}
So we proved \eqref{3.9}, obtaining a contradiction, which then forces $f(y)=a|y-a^{-1}|$ on $[0,\infty)$.

It remains to show that if $f\equiv 0$ on $[{a^{-1}},{a^{-1}}+\eps]$ for some $\eps>0$, then $f$ cannot be a minimizer of $H^+$ on $\calF^+$.  Assume first that $f\not\equiv 0$ on $[{a^{-1}},\infty)$, and let $p>a^{-1}$ be the largest number such that
$f\equiv 0$ on $[{a^{-1}},p]$.  Let also $q:=\min\{y>p\,|\, f(y)=0\}\in (p,\infty]$.  Then Lemma \ref{L.3.2}(iii) shows that there is at most one number in $(p,q)\cap P_0$, and if such number $y_0$ does exist, then  
\beq\lb{3.9a}
(p,q)\cap P_+=(p,y_0) \qquad\text{and}\qquad (p,q)\cap P_-=(y_0,q).
\eeq
  If there is no such $y_0$, then $(p,q)$ is one of $P_\pm$, so we can pick $y_0\in\{p,q\}$ so that \eqref{3.9a} holds.

Then Lemma \ref{L.3.2}(i,ii) show that either $f(y)= a(y-p)$ for all $y\in[p,\infty)$ (when $q=\infty$) or 
\[
f(y)= 
\begin{cases}
a(y-p) & y\in[p,\frac {q+p}2],
\\ a(q-y)  & y\in[\frac {q+p}2,q]
\end{cases}
\]
 (when $q<\infty$, and then we must have $\frac {q+p}2\ge y_0$).  In the latter case we let
\[
f_\delta(y):=
\begin{cases}
\max  \left\{ f(y), a(y-p) \right\}  & y\in[0,\frac {q+p}2+\delta],
\\ \max \{ f(y), a(q+2\delta-y)\}  & y> \frac {q+p}2+\delta,
\end{cases}
\]
and then
\[
0\le \lim_{\delta\to 0^+} \frac {H^+(f_{\delta}) - H^+(f)}{2a\delta} 
 = \int_{\frac {q+p}2}^{q}  \frac 1{y^3} \left( g \left(\frac {1+f(y)}y \right) - g \left(\frac {1-f(y)}y \right) \right) dy<0
\]
(the last inequality follows from $(\frac {q+p}2,q)\subseteq P_-$), a contradiction.  In the former case we let
\[
f_\delta(y):= \max\{f(y), a(y-p+\delta)\}
\]
and then
\begin{align*}
0 & \le \lim_{\delta\to 0^+} \frac {H^+(f_{\delta}) - H^+(f)}{a\delta} 
 = -\partial_B G^+(1-ap,p) + \partial_B G^-(1+ap,p) 
\\ & =-\frac 1{2a}\, \frac {p^2+2ap-1}{(p^2+1)^2} +  \frac 1{2a(1+a^2)}\, \frac {p^2(1-3a^2)+p(6a-2a^3)+3a^2-1}{(p^2+1)^2} 
\\ & =- \frac {2a}{1+a^2}\, \frac {p^2-Ap-1}{(p^2+1)^2} <0,
\end{align*}
where at the end we used $p>a^{-1}$, together with $a^{-2}-Aa^{-1}-1=0$ and $(y^2-Ay-1)'>0$ on $(a^{-1},\infty)$.
This is again a contradiction, so it remains to consider the case $f\equiv 0$ on $[a^{-1},\infty)$.

In that case we pick any $y_0>s_3^{-1}$, so that there is $\eps>0$ such that
\beq\lb{3.10}
\partial_{rr} h(y,r)=\frac 1{y^4} \, g' \left(\frac ry \right)<0
\eeq
for all $(y,r)\in B_{2\eps}((y_0,1))$.  Then we pick any $f_0\in \calF^+$ that is equal to $f$ outside $(y_0-\eps,y_0+\eps)$ and it is positive on this interval (from $f\in\calF^+$ we have $f_0\le a\eps$ there).  Then
\[
H^+(f_0)- H^+(f)=  \int_{y_0-\eps}^{y_0+\eps}  \big[ h(y, 1+f_0(y)) + h(y, 1-f_0(y)) - 2h(y, 1) \big] \, dy <0
\]
by \eqref{3.10}.  This again yields a contradiction with the assumption that $f$ is a minimizer of $H^+$ on $\calF^+$, so the proof of Lemma \ref{L.2.2} is finished.


\end{document}